\documentclass[a4paper]{article}
[12pt]

\makeatletter
\renewcommand*\l@section{\@dottedtocline{1}{1.5em}{2.3em}}
\makeatother

\usepackage{amsfonts}
\usepackage{amssymb}
\usepackage[T1]{fontenc}

\usepackage{tikz}
\usetikzlibrary{calc}

\usepackage{CJK}
\usepackage{amsmath}
 
\usepackage{amsfonts}
\usepackage{amssymb}
\usepackage{amsthm}
\usepackage{amssymb}
\usepackage{enumerate}
\usepackage[calc]{picture}
\usepackage[all,cmtip]{xy}

\usepackage[mathscr]{eucal}
\usepackage{eqlist}

\usepackage{color}
\usepackage{abstract} 
\usepackage[T1]{fontenc}
 
\usepackage[top=3cm, bottom=3cm, left=3cm, right= 3cm]{geometry}

\theoremstyle{plain}
\newtheorem{theorem}{Theorem}
\newtheorem{proposition}[theorem]{Proposition}
\newtheorem{lemma}[theorem]{Lemma}

\newtheorem{example}[theorem]{Example}
\newtheorem{corollary}[theorem]{Corollary}

\theoremstyle{definition}
\newtheorem{definition}{Definition}

\usepackage{etoolbox}
\newtheoremstyle{myrem}
 {3pt}
 {3pt}
 {\normalsize}
 { }
 {\itshape}
 {:}
 { }
 {}

 \theoremstyle{myrem}
 \newtheorem{remark}{Remark}
 \appto\remark{\leftskip\parindent}
 \appto\remark{\rightskip\parindent}

\numberwithin{equation}{section}
\numberwithin{theorem}{section}

\begin{document}

\begin{center}
{\Large {\textbf {Weighted Path homology of Weighted Digraphs and 
Persistence}}}
 \vspace{0.58cm}\\

 Chong Wang*\dag,  Shiquan Ren*,   Jie Wu*,  Yong Lin*

\bigskip



\footnotetext[3]
{* first authors;   \dag   corresponding author. }

\bigskip

\begin{quote}
\begin{abstract}
\smallskip

 In recent years,  A. Grigor'yan, Y. Lin, Y. Muranov and S.T. Yau  \cite{9,10,yau3,yau4}  constructed a path homology theory for digraphs.  Later,  S. Chowdhury and F. M\'{e}moli  \cite{11} studied the persistent path homology for directed networks.  In this paper,  we generalize the path homology theory for digraphs and construct a weighted path homology for weighted digraphs.  We study the persistent weighted path homology for weighted digraphs and detect the effects of the weights on the persistent weighted path homology.  We prove a persistent version of a K\"{u}nneth-type formula for joins of weighted digraphs.  
\end{abstract}
\end{quote}

\end{center}

\bigskip


\section{Introduction}\label{s1}


Directed graphs, or digraphs for short, is an important mathematical model for describing complex networks in information science.  A digraph $G$ is a pair $(V,E)$ where $V$ is a set, and $E$ is a subset of all the ordered  pairs $(u,v)\in V\times V$ such that $u\neq v$. Note that ($u, v)\in E$ implies  that $u, v$ are distinct.   The elements in $V$ are called vertices and the elements in $E$ are called directed edges.    Throughout this paper, we assume that  $V$ is finite, thus $E$ is finite as well.

So far, various approaches to construct (co)homology theories for digraphs have been studied. 
  For example, D. Happel \cite{23} studied the Hochschild homology of finite dimensional algebras, which can be applied to the path algebra of digraphs. However,  it was shown in \cite{23} that the Hochschild homologies of order greater than one are trivial, which makes this approach less attractive.   Another example is that the homologies constructed by H. Barcelo in \cite{2} can be applied to digraphs.  However, these homologies may not be functorial with respect to morphisms of digraphs.  

In recent years,  A. Grigor'yan, Y. Lin, Y. Muranov and S.T. Yau  \cite{9,10,yau3,yau4}  constructed a path homology theory for digraphs.   They construct a chain complex by taking certain linear combinations  of all the paths (called allowed elementary paths) on the digraphs, along with the directions of the edges.  For $k\geq 0$,  the face maps $d_k$ are given by deleting the $k$-th single vertice in the paths, and the boundary maps  are the alternating sums of the $d_k$'s.  By the examples in \cite{9,yau4},  the path homology of digraphs is highly non-trivial.  Moreover, it is proved in \cite{10} that  the path homology is functorial with respect to morphisms of digraphs, and is invariant up to certain homotopy relations of these  morphisms.

Based on the path homology theory given in \cite{9,10,yau3,yau4}, some K\"{u}nneth formulas have been proved by A. Grigor'yan, Y. Muranov and S.T. Yau in \cite{3}, and a path homology theory for hypergraphs has been constructed by   A. Grigor'yan, 
R.  Jimenez, 
 Y. Muranov and 
 S.T. Yau in \cite{hd}.  Moreover,  as applications of the path homology theory,  the persistent path homology for directed networks has been investigated by  S. Chowdhury and F. M\'{e}moli in \cite{11}.

On the other hand,   the homology theory of  simplicial complexes has been generalized to weighted homology theory for weighted simplicial complexes by R.J. MacG. Dawson \cite{weight1} in 1990,  and later studied by     D. Horak and J. Jost \cite{adv1},   S. Ren, C. Wu and J. Wu \cite{27,28}, etc.   Given an (abstract) simplicial complex,  one may  assign a weight $w(i)$ to each vertex $i$,  then define the weighted boundary operators as the linear combination
\begin{eqnarray}\label{eq-0.0.0}
\sum_{k\geq 0} (-1)^k w(i_k) d_k
\end{eqnarray}
 where $i_k$ is the $k$-th vertex of a simplex.   
It is proved in \cite{weight1,adv1,28} that the graded free module generated by the simplices equipped with the the weighted boundary operators gives a chain complex. Hence the weighted homology is well-defined.

\smallskip

In this paper,  we borrow the idea of weighted homology to generalize the path homology theory of digraphs.  We  construct a weighted path homology theory for weighted digraphs and study the persistence.  We prove a K\"{u}nneth-type formula for the weighted path homology of joins of weighted digraphs.

For each vertex $i$ in a digraph $G$, we assign a weight $w(i)$. We define the weighted boundary operators formally as the linear combination (\ref{eq-0.0.0})  where $i_k$ is the $k$-th vertex of a path and $d_k$ is the face map deleting the $k$-th vertex of a path.  In Section~\ref{s333},  we prove that the weighted path homology is well-defined and is functorial up to morphisms of weighted digraphs. In Section~\ref{s444},   we study the persistence of weighted path homology.  We give an isomorphism result for the persistent weighted path homology in Theorem~\ref{th-5.z}.  We prove  in Theorem~\ref{th-5.w} that   with field coefficients, the persistent weighted path homology is the same as the usual persistent path homology  studied in \cite{11};  and  we give some examples Subsection~\ref{ss4.3}  to show that with integral coefficients, the persistent weighted path homology depends on the weights.  Finally,  in Section~\ref{sss5},  we prove a persistent version of the K\"{u}nneth-type formula for the weighted path homology of joins of weighted digraphs.

{

\section{Preliminaries: path homology of digraphs}\label{s1}

In this section,  we review some backgrounds in \cite{9,10} about paths on digraphs,   the path homology groups of digraphs,  morphisms of digraphs,   and the induced homomorphisms of  path homology groups.  We make a slight generalization of \cite{9,10} by taking the coefficients in  a commutative ring $R$ with unit, instead of in a field.

\smallskip

\subsection{Paths on finite sets and chain complexes}

Let $V$ be a nonempty finite set whose elements will be called vertices. For any  integer $p\geq -1$, an {\it elementary $p$-path} on  $V$ is an (ordered) sequence $\{i_k\}_{k=0}^p$ consisting of $p+1$ vertices in $V$. For convenience,  $\{i_k\}^p_{k=0}$  will  be alternatively denoted as $e_{i_0 \cdots i_p}$.   An elementary $(-1)$-path on $V$ is defined as the empty set $\emptyset$.   An elementary $p$-path $e_{i_0\cdots i_{p}}$ on a set $V$ is called {\it regular}  if $i_{k}\not=i_{k+1}$ for all $k=0,\cdots,{p-1}$, and  {\it non-regular}  otherwise.

Let $R$ be a commutative ring with unit. Consider the free $R$-module $\Lambda_p(V)$  consisting of all the formal linear combinations of elementary $p$-paths on $V$ with  coefficients in $R$.  The elements of $\Lambda_p(V)$ are called {\it $p$-paths} on $V$.  
The set $\{e_{i_0\cdots i_p}: i_0, \cdots, i_p\in V\}$ is a basis in $\Lambda_p(V)$. Hence each $p$-path $v\in \Lambda_p(V)$ has a unique representation in the form
\begin{eqnarray*}
v=\sum_{i_0, \cdots, i_p\in V}v^{i_0 \cdots i_p}e_{i_0 \cdots i_p},
\end{eqnarray*}
where $v^{i_0 \cdots i_p}\in  R$ are the coefficients in $R$. For example, $\Lambda_0(V)$ consists of all the linear combinations of the elements $e_i$ where $i\in V$, $\Lambda_1(V)$ consists of all the linear combinations of the elements $e_{ij}$ where $(i,j)\in V\times V$, and so on.  Note that $\Lambda_{-1}(V)$ consists of all multiples of the unit $e\in R$ so that $\Lambda_{-1}(V)= R$.  
 For any $p\geqslant{0}$, we define the boundary operator 
 \begin{eqnarray}\label{eq-2.1}
 \partial:\Lambda_p(V)\longrightarrow\Lambda_{p-1}(V)
 \end{eqnarray}
   by setting 
 \begin{eqnarray*}
 \partial(e_{i_0\cdots i_p})= \sum_{q=0}^{p}(-1)^q d_q (e_{i_0\cdots i_p}),
 \end{eqnarray*}
where 
\begin{eqnarray*}
d_q (e_{i_0\cdots i_p})=e_{i_0\cdots \widehat{i _q}\cdots i_p}
\end{eqnarray*}
 are the face maps,  and 
$\widehat{i _q}$ means omission of the vertex $i_q$.  
 For example, 
\begin{eqnarray*}
\partial(e_{i})=e,   ~~~\partial({e_{ij}})=e_j-e_i,  ~~~\partial{e_{ijk}}=e_{jk}-e_{ik}+e_{ij}.
\end{eqnarray*}
 We extends $\partial$ linearly over $R$ and obtain the map (\ref{eq-2.1}).   It follows from  \cite[Section~2.1]{9}  that  $\partial^2=0$.  

For each $p\geq -1$, let $\mathcal{R}_{p}(V)$ be the sub-$R$-module of $\Lambda_{p}(V)$ generated by all the regular elementary $p$-paths, and let $\mathcal{I}_{p}(V)$ be the sub-$R$-module of $\Lambda_{p}(V)$ generated by all the non-regular elementary $p$-paths.  By \cite[Section~2.3]{9}, we have the $R$-module isomorphism
\begin{eqnarray}\label{eq-2.2}
\mathcal{R}_{p}(V)\cong \Lambda_{p}(V)/\mathcal{I}_{p}(V). 
\end{eqnarray}
By \cite[Lemma~2.9]{9}, the boundary operator (\ref{eq-2.1}) induces a boundary operator on $\Lambda_{p}(V)/\mathcal{I}_{p}(V)$.  By the isomorphism (\ref{eq-2.2}),  we have an induced boundary operator 
\begin{eqnarray}\label{eq-2.3}
\partial:  \mathcal{R}_{p}(V)\longrightarrow \mathcal{R}_{p-1}(V)
\end{eqnarray} 
as the pull-back of the boundary operator on $\Lambda_{p}(V)/\mathcal{I}_{p}(V)$.   Consequently,  (\ref{eq-2.3}) gives a chain complex 
\begin{eqnarray}\label{eq-2.5}
\cdots \overset{\partial}{\longrightarrow} \mathcal{R}_{p}(V)\overset{\partial}{\longrightarrow} \mathcal{R}_{p-1}(V)\overset{\partial}{\longrightarrow}\cdots \overset{\partial}{\longrightarrow}\mathcal{R}_{0}(V) \overset{\partial}{\longrightarrow} R \overset{\partial}{\longrightarrow} 0. 
\end{eqnarray}
In the remaining part of this paper, we always use $\partial$ to denote the boundary operator of the chain complex (\ref{eq-2.5}) if there is no extra claim.

\smallskip

\subsection{Paths on digraphs and the path homology}

Let $G=(V,E)$ be a digraph.  An elementary $p$-path $e_{i_0 \cdots i_p} \in \Lambda_{p}(V)$ is called {\it allowed} on $G$ if $i_{k-1}\to i_k$ for all $k=1,2,...,p$.  An allowed elementary $p$-path $e_{i_0\cdots i_p}$ is called  {\it closed} if $i_0 =i_p$.   Let $\mathcal{A}_p(G)$ be the free $R$-module generated by all the  allowed elementary $p$-paths in $G$.   Then  $\mathcal{A}_p(G)$  is a sub-$R$-module of  $\mathcal{R}_{p}(V)$  consisting of all the formal linear combinations of the  allowed elementary $p$-paths $e_{i_0\cdots i_p}$ with coefficients in $R$.  Note that under the boundary operator $\partial$,  the image  of an allowed path does not have to be allowed.  Hence $\partial$ may not map $\mathcal{A}_p(G)$  to $\mathcal{A}_{p-1}(G)$.   
Nevertheless,  $\mathcal{A}_p(G)$ has the sub-$R$-module 
\begin{eqnarray*}
\Omega_p(G)=\{v\in\mathcal{A}_p(G) : \partial{v}\in\mathcal{A}_{p-1}(G)\}
\end{eqnarray*}
whose elements are called {\it $\partial$-invariant $p$-paths}, 
which satisfies $\partial\Omega_p(G)\subset\Omega_{p-1}(G)$ for all $p\geq{-1}$ (cf. \cite[Section~3.3]{9}).  Hence (\ref{eq-2.5}) has a sub-chain complex
\begin{eqnarray}\label{eq-2.6}
\cdots \overset{\partial}{\longrightarrow} \Omega_{p}(G)\overset{\partial}{\longrightarrow}\Omega_{p-1}(G)\overset{\partial}{\longrightarrow}\cdots \overset{\partial}{\longrightarrow}\Omega_{0}(G) \overset{\partial}{\longrightarrow} R \overset{\partial}{\longrightarrow} 0. 
\end{eqnarray}
The truncated version of (\ref{eq-2.6}) is the chain complex by replacing $\Omega_{0}(G) \overset{\partial}{\longrightarrow} R$ with $\Omega_{0}(G) \overset{\partial}{\longrightarrow} 0$ in (\ref{eq-2.6}).  The homology groups 
\begin{eqnarray*}
H_p(G;R)=\text{Ker} (\partial \mid_{\Omega_{p}})/\text{Im} (\partial\mid_{\Omega_{p+1}}), ~~~ p\geq 0  
\end{eqnarray*}
of the truncated version of (\ref{eq-2.6})  are called the {\it path homology groups} of the digraph $G$ (cf. \cite[Definition~3.12]{9}).
As a concrete example, the construction of the chain complex (\ref{eq-2.6}) and the calculation of the path homology of digraphs are illustrated in \cite[Example~10]{11}.

\smallskip

\subsection{Morphisms of digraphs and induced homomorphisms between path homologies}

Let $G=(V,E)$ and $G'=(V',E')$ be two digraphs.   
By \cite[Definition~2.2]{10}, 
a  {\it  morphism of digraphs} is a map $f: V\longrightarrow V'$ such that whenever $i\to j$ is a directed edge in $E$,  we have either 
$f(i)\to f(j)$ is a directed edge in $E'$, 
 or 
$f(i)=f(j)$.  We write such a morphism as $f: G\longrightarrow G'$.

Let $f: G\longrightarrow G'$ be a morphism of digraphs.  For each $p\geq -1$, we have an induced homomorphism of $R$-modules
 \begin{eqnarray}\label{eq-4.1}
 f_{\#\#\#}: \Lambda_p(V)\longrightarrow \Lambda_p(V')  
 \end{eqnarray}
which sends $e_{i_0\cdots i_p}$ to $e_{f(i_0)\cdots f(i_p)}$ for any elementary $p$-path $e_{i_0\cdots i_p}$ on $V$. 
We observe that $f_{\#\#\#}$ in (\ref{eq-4.1}) sends non-regular elementary $p$-paths on $V$ to non-regular elementary $p$-paths on $V'$.  Hence in (\ref{eq-4.1}), $f_{\#\#\#}(\mathcal{I}_p(V))\subseteq \mathcal{I}_p(V')$.  Thus we have a quotient homomorphism of $R$-modules
\begin{eqnarray}\label{eq-4.2}
f_{\#\#}:    \Lambda_{p}(V)/\mathcal{I}_{p}(V)\longrightarrow \Lambda_{p}(V')/\mathcal{I}_{p}(V'). 
\end{eqnarray}
With the help of the isomorphism (\ref{eq-2.2}),  we have that (\ref{eq-4.2}) induces a morphism of $R$-modules
\begin{eqnarray}\label{eq-4.3}
f_{\#}: \mathcal{R}_{p}(V)\longrightarrow \mathcal{R}_{p}(V'). 
\end{eqnarray}
Let $\partial$ and $\partial'$ be the boundary operators of $\mathcal{R}_{*}(V)$ and $\mathcal{R}_{*}(V')$ respectively.  It follows from a straight-forward calculation that 
\begin{eqnarray}\label{eq-4.6}
\partial' f_{\#}=f_{\#}\partial. 
\end{eqnarray}
Hence $f_{\#}$ is a chain map. Moreover,  by \cite[Theorem~2.10]{10},  the restriction of $f_{\#}$ gives  a chain map
\begin{eqnarray}\label{eq-4.a1}
f_{\#}:  \Omega_p(G)\longrightarrow\Omega_p(G'), ~~~ p\geq 0, 
\end{eqnarray}
which sends an allowed elementary $p$-path $e_{i_0\cdots i_p}$   to $e_{f(i_0)\cdots f(i_p)}$ whenever $e_{f(i_0)\cdots f(i_p)}$ is regular, and sends the allowed elementary $p$-path $e_{i_0\cdots i_p}$ to $0$ otherwise (whenever $e_{f(i_0)\cdots f(i_p)}$ is non-regular).  
Therefore, (\ref{eq-4.a1}) induces a homomorphism between the path homologies
\begin{eqnarray}\label{eq-4.a2}
f_*: H_p(G;R)\longrightarrow H_p(G';R), ~~~ p\geq 0. 
\end{eqnarray}

\section{Weighted path homology of weighted digraphs}\label{s333}

In this section,  we generalize the path homology of digraphs in \cite{9} and construct a weighted path homology for weighted digraphs.  We define morphisms of weighted digraphs.  As a generalization of \cite[Theorem~2.1]{10},  we prove that such a morphism induces homomorphisms between the weighted path homologies, in Theorem~\ref{th-4.3}.

\smallskip

\subsection{The construction of weighted path homology}

Let $G=(V,E)$ be a digraph.  Let $R$ be a commutative ring with unit. 

\begin{definition}\label{def-3.1}
A  {\it weight} on   $G$ is a function function $w: V\longrightarrow R$.  We call $G$ equipped with $w$  a {\it weighted digraph} and denote the couple  as $(G,w)$. 
\end{definition}

Let $w$ be a weight on $G$.  
We define the {\it weighted boundary operator }  
\begin{eqnarray}\label{eq-3.11}
\partial^w:  \Lambda_p(V)\longrightarrow\Lambda_{p-1}(V)
\end{eqnarray}
by setting
\begin{eqnarray*}
\partial^w (e_{i_0\cdots i_p})= \sum_{q=0}^{p}(-1)^q w(i_q) d_q (e_{i_0\cdots i_p})
\end{eqnarray*}
and extending linearly over $R$.  The next lemma follows from a straight-forward calculation.  

\begin{lemma}\label{pr-3.1}
$\partial^{w} \circ \partial^{w}=0$.
\end{lemma}
\begin{proof}
Let $e_{i_0...i_p}$ be an elementary $p$-path. Then
\begin{eqnarray*}
\partial^{w} \circ \partial^{w} (e_{i_0...i_p})&=& \partial^{w} \sum_{q=0}^{p}(-1)^q w(i_q)  (e_{i_0\cdots \widehat{i_q} \cdots i_p})\\
&=& \sum_{0\leq{r}<q\leq{p}}(-1)^{r+q} {w(i_r)w(i_q)(e_{i_0\cdots \widehat{i_r} \cdots \widehat{i_q} \cdots i_p})}\\
&&+\sum_{0\leq{q}<r\leq{p}}(-1)^{r+q-1} {w(i_r)w(i_q)(e_{i_0\cdots \widehat{i_q} \cdots \widehat{i_r} \cdots i_p})}. 
\end{eqnarray*}
After switching $r$ and $q$ in the  sums, we see that the two sums in the last equality cancel out. Hence
 $\partial^{w} \circ \partial^{w}(e_{i_0...i_p})=0$.  Since $\partial^w$ is linear over $R$,  it follows that $\partial^{w} \circ \partial^{w}(u)= 0$ for all $u\in\Lambda_p(V)$. The lemma is proved. 
\end{proof}

The next lemma is a generalization of \cite[Lemma~2.9]{9}~(a). 

\begin{lemma}\label{le-3.2}
Let $p\geq -1$.  Suppose the weight function $w$ is non-vanishing on $V$.  
If $v_1,v_2\in \Lambda_p(V)$ and $v_1=v_2$ mod $\mathcal{I}_p(V)$,  then $\partial^w v_1=\partial^w v_2$ mod $\mathcal{I}_{p-1}(V)$. 
\end{lemma}

\begin{proof}
Without loss of generality, we may assume $p\geq 1$. Similar with the proof of  \cite[Lemma~2.9]{9}~(a), it suffices to prove that if $e_{i_0\cdots i_p}=0$ mod $\mathcal{I}_p(V)$,  then  $\partial^w e_{i_0\cdots t_p}=0$ mod $\mathcal{I}_{p-1}(V)$.   Let $e_{i_0\cdots i_p}$ be non-regular such that $i_k=i_{k+1}$ for some $0\leq k\leq p-1$.  Then 
\begin{eqnarray}\label{eq-3.27}
\partial^w  e_{i_0\cdots t_p}=\sum_{0\leq r\leq p,\atop r\neq k,k+1}(-1)^r w(i_r) e_{i_0 \cdots \widehat{i_r} \cdots i_p}.  
\end{eqnarray}
Since $w$ is non-vanishing,  $w(i_r)\neq 0$ for each $0\leq r\leq p$ and  $r\neq k,k+1$. Hence  by (\ref{eq-3.27}), $\partial^w  e_{i_0\cdots t_p}$ is non-regular as well. The lemma follows. 
\end{proof}
 
 \smallskip
 
 In the remaining part of this paper,  we assume that $w$ is non-vanishing on $V$ without extra claims.   
By Lemma~\ref{le-3.2},  the weighted boundary operator (\ref{eq-3.11}) induces a weighted boundary operator on $\Lambda_{p}(V)/\mathcal{I}_{p}(V)$.  By the isomorphism (\ref{eq-2.2}),  we have an induced weighted boundary operator 
\begin{eqnarray}\label{eq-3.13}
\partial^w:  \mathcal{R}_{p}(V)\longrightarrow \mathcal{R}_{p-1}(V)
\end{eqnarray} 
as the pull-back of the weighted boundary operator on $\Lambda_{p}(V)/\mathcal{I}_{p}(V)$.  Consequently,  (\ref{eq-3.13}) gives a chain complex 
\begin{eqnarray}\label{eq-3.15}
\cdots \overset{\partial^w}{\longrightarrow} \mathcal{R}_{p}(V)\overset{\partial^w}{\longrightarrow} \mathcal{R}_{p-1}(V)\overset{\partial^w}{\longrightarrow}\cdots \overset{\partial^w}{\longrightarrow}\mathcal{R}_{0}(V) \overset{\partial^w}{\longrightarrow} R \overset{\partial^w}{\longrightarrow} 0. 
\end{eqnarray}
We consider the sub-$R$-module 
\begin{eqnarray}\label{eq-4.000}
\Omega^w_p(G)=\{v\in\mathcal{A}_p(G) : \partial^w{v}\in\mathcal{A}_{p-1}(G)\}. 
\end{eqnarray}
By a similar argument with \cite[Section~3.3]{9},  we have 
$\partial^w\Omega^w_p(G)\subset\Omega^w_{p-1}(G)$.  
Hence (\ref{eq-3.15}) has a sub-chain complex
\begin{eqnarray}\label{eq-3.16}
\cdots \overset{\partial^w}{\longrightarrow} \Omega^w_{p}(G)\overset{\partial^w}{\longrightarrow}\Omega^w_{p-1}(G)\overset{\partial^w}{\longrightarrow}\cdots \overset{\partial^w}{\longrightarrow}\Omega^w_{0}(G) \overset{\partial^w}{\longrightarrow} R \overset{\partial^w}{\longrightarrow} 0. 
\end{eqnarray}
We define the {\it weighted path homology}  of the weighted digraph  $(G,w)$   as the homology groups of the truncated version of  (\ref{eq-3.16}):
\begin{eqnarray}\label{eq-4.c3}
H_p(G,w;R)=\text{Ker} (\partial^w \mid_{\Omega^w_{p}})/\text{Im} (\partial^w\mid_{\Omega^w_{p+1}}), ~~~ p\geq 0.   
\end{eqnarray}
It follows from (\ref{eq-4.000}) and (\ref{eq-4.c3}) that
\begin{eqnarray}\label{eq-4.00}
H_p(G,w;R)=\mathcal{A}_p(G)\cap \text{Ker}(\partial^w) /  \mathcal{A}_p(G)\cap \partial^w\mathcal{A}_{p+1}(G). 
\end{eqnarray}

\smallskip

Finally,  we give an alternative construction  for the weighted path homology.  We consider the chain complex $\{R_*(V),\partial^w\}$  given in (\ref{eq-3.15}) and its graded sub-$R$-module $\mathcal{A}^w_*(G)$.    We let
\begin{eqnarray*}
\Gamma^w_p(G)= \mathcal{A}^w_p(G)+ \partial^w \mathcal{A}^w_{p+1}(G), ~~~ p\geq 0. 
\end{eqnarray*}
It can be proved that $\{\Gamma^w_*(G),\partial^w\}$ is a chain complex, whose homology is isomorphic to the weighted path homology $H_*(G,w;R)$.  In fact, we have the following proposition. 
 
 \begin{proposition}\label{le-5.a}
 Equipped with the boundary operator $\partial^w$, $\Omega^w_*(G)$ is the largest sub-chain complex of  $\{R_*(V),\partial^w\}$  that is contained in $\mathcal{A}^w_*(G)$ as  graded sub-$R$-modules;  and $\Gamma^w_*(G)$ is the smallest sub-chain complex of $\{R_*(V),\partial^w\}$ containing $\mathcal{A}^w_*(G)$  as graded sub-$R$-modules.  Moreover, the canonical inclusion $\iota: \Omega^w_*(G)\longrightarrow \Gamma^w_*(G)$ induces an isomorphism of homologies.  
 \end{proposition}

\begin{proof}
 The proof follows by applying \cite[Section~2]{hypergraph} to the graded sub-$R$-module $\mathcal{A}^w_*(G)$
 in the chain complex $\{R_*(V),\partial^w\}$.  The infimum chain complex is $\Omega^w_*(G)$, and the supremum chain complex is $\Gamma^w_*(G)$.  By \cite[Proposition~2.4]{hypergraph},  the canonical inclusion $\iota$ induces an isomorphism in homology.  
 \end{proof}
 
 By Proposition~\ref{le-5.a}, the weighted path homology of a weighted digraph $(G,w)$ can be constructed either as the homology of $\Omega_*^w(G)$ or as the homology of $\Gamma_*^w(G)$
 \begin{eqnarray*}
H_p(G,w;R)\cong \text{Ker} (\partial^w \mid_{\Gamma^w_{p}})/\text{Im} (\partial^w\mid_{\Gamma^w_{p+1}}), ~~~ p\geq 0.    
\end{eqnarray*}


\smallskip

\subsection{Morphisms of weighted digraphs and induced homomorphisms between weighted path homologies}


Let $w$ be a weight on $V$ and $w'$ be a weight on $V'$.  Then we have weighted digraphs $(G,w)$ and $(G',w')$.   
\begin{definition}\label{def-4.2}
A {\it morphism of weighted digraphs}   from $(G,w)$ to $(G',w')$    is a morphism of digraphs $f: G\longrightarrow G'$ such that for any $i\in V$, $w(i)=w'(f(i))$.   
 \end{definition}
 
 Suppose that  $f: (G,w)\longrightarrow (G',w')$ is a morphism of weighted digraphs.  The next lemma is a generalization of (\ref{eq-4.6}). 
 
 \begin{lemma}\label{le-4.1}
We have
$\partial^{'w'} f_{\#}=f_{\#}\partial^w$.  Hence (\ref{eq-4.3}) gives a chain map between the chain complexes  $\{\mathcal{R}_{*}(V),\partial^w\}$ and $\{\mathcal{R}_{*}(V'),\partial^{'w'}\}$,  with weighted boundary operators. 
 \end{lemma}
 
 \begin{proof}
Since $f_\#$ is canonically induced from $f_{\#\#\#}$,  it suffices to prove that $f_{\#\#\#}$ in (\ref{eq-4.1})  is a chain map with respect to the weighted boundary operators $\partial^w$ and $\partial^{'w'}$.   Let $e_{i_0\cdots i_p}$ be an elementary $p$-path in $\Lambda_p(V)$.  Then
\begin{eqnarray*}
\partial^{'w'} f_{\#\#\#}(e_{i_0\cdots i_p})&=& \partial^{'w'}  e_{f(i_0) \cdots f(i_p)}\\
&=& \sum_{q=0}^{p}(-1)^q w'(f(i_q)) d_q (e_{f(i_0)\cdots f(i_p)})\\
&=& \sum_{q=0}^{p}(-1)^q w(i_q) d_q (e_{f(i_0)\cdots f(i_p)})\\
&=&f_{\#\#\#}\partial^w (e_{i_0\cdots i_p}). 
\end{eqnarray*}
 Therefore,  $f_{\#\#\#}$ in (\ref{eq-4.1})  is a chain map with respect to the weighted boundary operators $\partial^w$ and $\partial^{'w'}$. Thus $f_\#$ is a chain map as well. 
 \end{proof}

 The next lemma, which  
 is a genralization of (\ref{eq-4.a1}),  follows   from Lemma~\ref{le-4.1}.  
 
 \begin{lemma}\label{le-4.2}
 Restricted to the sub-chain complexes  $\Omega^w_*(G)$ and $\Omega^{w'}_*(G')$   
 of $\mathcal{R}_{*}(V) $ and $\mathcal{R}_{*}(V')$ respectively, (\ref{eq-4.3}) gives  an induced chain map from $\{\Omega^w_*(G),\partial^w\}$ to $\{\Omega^{w'}_*(G'),\partial^{'w'}\}$. 
 \end{lemma}
 
 \begin{proof}
 By a similar argument of \cite[(2.10) in the proof of Theorem~2.10]{10},   we have
 \begin{eqnarray}\label{eq-4.c1}
 f_\#(\Omega^w_p(G))\subseteq  \Omega^{w'}_p(G'), ~~~ p\geq 0. 
 \end{eqnarray}
 In fact, for any $v\in \Omega_p^w(G)$, we have $v\in \mathcal{A}_p(G)$ and $\partial^w v\in \mathcal{A}_{p-1}(G)$. Hence $f_\#(v)\in \mathcal{A}_p(G')$ and by Lemma~\ref{le-4.1}, 
 \begin{eqnarray*}
 \partial^{'w'}( f_\#(v)) =f_\# \partial^wv\in f_\#(\mathcal{A}_{p-1} (G))\subseteq \mathcal{A}_{p-1}(G'). 
 \end{eqnarray*}
 Consequently, $v\in \Omega^{w'}_p(G')$ and we obtain (\ref{eq-4.c1}).   
 Hence the restriction of $f_\#$ to $\Omega^w_*(G)$ is a chain map with respect to the weighted boundary operators $\partial^w$ and $\partial^{'w'}$.  
 \end{proof}
 
 \begin{remark}\label{re-4.1}
Similarly to the proof of Lemma~\ref{le-4.2}, we can obtain the following statement: {\it restricted to the sub-chain complexes  $\Gamma^w_*(G)$ and $\Gamma^{w'}_*(G')$   
 of $\mathcal{R}_{*}(V) $ and $\mathcal{R}_{*}(V')$ respectively, (\ref{eq-4.3}) gives  an induced chain map from $\{\Gamma^w_*(G),\partial^w\}$ to $\{\Gamma^{w'}_*(G'),\partial^{'w'}\}$}. 
 \end{remark}
 
 \smallskip

With the help of Lemma~\ref{le-4.2}, we have  the next theorem,  which is a generalization of (\ref{eq-4.a2}). 

\begin{theorem}\label{th-4.3}
Let $(G,w)$ and $(G',w')$ be two weighted digraphs. Let $f: (G,w)\longrightarrow (G',w')$ be a morphism of weighted digraphs. Then $f$ induces a homomorphism between the weighted path homologies
\begin{eqnarray}\label{eq-4.21}
f_*: H_p (G,w;R)\longrightarrow H_p(G',w';R), ~~~ p\geq 0. 
\end{eqnarray}
\end{theorem}

\begin{proof}
By Lemma~\ref{le-4.2}, the chain map $f_\#$ from $\{\Omega^w_*(G),\partial^w\}$ to $\{\Omega^{w'}_*(G'),\partial^{'w'}\}$ induces a homomorphism in homology
\begin{eqnarray}\label{eq-4.c6}
f_*: H_p(\{\Omega^w_*(G),\partial^w\})\longrightarrow H_p(\{\Omega^{w'}_*(G'),\partial^{'w'}\}), ~~~ p\geq 0. 
\end{eqnarray}   
By the definition of the weighted path homology  (\ref{eq-4.c3}) of weighted digraphs,   the homomorphism (\ref{eq-4.c6}) gives (\ref{eq-4.21}). 
\end{proof}

\begin{remark}\label{re-4.2}
 Theorem~\ref{th-4.3}  can be alternatively proved by applying Remark~\ref{re-4.1} instead of applying Lemma~\ref{le-4.2}.  
\end{remark}

\smallskip
 
 Finally, as  supplements to Theorem~\ref{th-4.3}, we list some properties of the induced homomorphisms between the weighted path homologies:  
 
 \begin{enumerate}[(I).]
 \item
 the identity morphism $\text{id}$ from a weighted digraph $(G,w)$ to itself   induces the identity map $\text{id}_*$ on  $H_*(G,w;R)$;
 
 \item
 given two morphisms $f: (G,w)\longrightarrow (G',w')$ and $g: (G',w')\longrightarrow (G'',w'')$ of weighted digraphs, 
 the induced homomorphism $(g\circ f)_*$ of the composition $g\circ f$ is the composition of  induced homomorphisms $g_*\circ f_*$. 
 \end{enumerate}

 \begin{proof}[Proof of (I) and (II)]
  The identity map $\text{id}$ of the  weighted digraph $(G,w)$, where $G=(V,E)$, induces the identity map  $\text{id}_\#$  on the chain complex $\{R_*(V),\partial^w\}$.  And the restriction of $\text{id}_\#$ on the sub-chain complex $\Omega_*^w(G)$ induces the identity map on the weighted path homology. We obtain (I).

Suppose $G'=(V',E')$ and $G''=(V'',E'')$.  The composition $g\circ f$ induces a chain map 
 \begin{eqnarray*}
(g\circ f)_{\#\#\#}:  \{\Lambda_*(V),\partial^w\}\overset{f_{\#\#\#}}{\longrightarrow} \{\Lambda_*(V'),\partial^{'w'}\} \overset{g_{\#\#\#}}{\longrightarrow} \{\Lambda_*(V''),\partial^{w''}\}. 
 \end{eqnarray*}
And   $(g\circ f)_{\#\#\#}$ induces  a chain map
\begin{eqnarray*}
(g\circ f)_\#: \{\mathcal{R}_{*}(V),\partial^w\}\overset{f_{\#}}{\longrightarrow} \{\mathcal{R}_{*}(V'),\partial^{'w'}\} \overset{g_{\#}}{\longrightarrow} \{\mathcal{R}_{*}(V''),\partial^{w''}\}, 
\end{eqnarray*}
which gives the restriction to  sub-chain complexes   
\begin{eqnarray}\label{eq-4.q1}
(g\circ f)_\#: \{\Omega_{*}(G),\partial^w\}\overset{f_{\#}}{\longrightarrow} \{\Omega_{*}(G'),\partial^{'w'}\} \overset{g_{\#}}{\longrightarrow} \{\Omega_{*}(G''),\partial^{w''}\}.  
\end{eqnarray}
The chain map $(g\circ f)_{\#}$ in (\ref{eq-4.q1}) induces a homomorphism in homology
\begin{eqnarray*}
(g\circ f)_*:H_p(G,w;R) \overset{f_*}{\longrightarrow} H_p(G',w';R) \overset{g_*}{\longrightarrow} H_p(G'',w'';R). 
\end{eqnarray*}
Thus we obtain (II). 
 \end{proof}

 \section{Persistent weighted path homology}\label{s444}
 
 In this section,  we discuss the persistent weighted path homology of weighted digraphs and morphisms of weighted digraphs, with ring coefficients.  We prove an isomorphism of the persistent weighted path homology in  Theorem~\ref{th-5.z}. In addition,  if the coefficients of the homology are in a field, then we also prove that the persistent weighted homology does not depend on the weights, in Theorem~\ref{th-5.w}.  Thus with field coefficients, the persistent weighted path homology is the same as the usual persistent path homology studied in \cite{11}.  Moreover, we give some examples in Subsection~\ref{ss4.3}.

 \smallskip
 
 \subsection{The inclusion of persistence complexes induces the identity map of persistent weighted homology}

 Let  $n=1,2,\ldots$.   
 We consider a finite or countable sequence of weighted digraphs $(G_n,w_n)$, together with a sequence of morphisms of weighted digraphs $f_n: (G_n,w_n)\longrightarrow (G_{n+1},w_{n+1})$.  Then by Lemma~\ref{le-4.2},  we have a persistence complex (we refer to \cite[Defiition~3.1]{[25-ZC05]} for the definition of persistence complexes)
 
 \begin{eqnarray*}
 \xymatrix{
\Omega_0^{w_1}(G_1)\ar[rr]^{{(f_1)}_\#} &&
    \Omega_0^{w_2}(G_2)\ar[rr]^{{(f_2)}_\#}&&
    \cdots 
     \\
     \\
 \Omega_1^{w_1}(G_1)\ar[rr]^{{(f_1)}_\#}\ar[uu]^{\partial^{w_1}}&&
    \Omega_1^{w_2}(G_2)\ar[rr]^{{(f_2)}_\#}\ar[uu]^{\partial^{w_2}}&&
    \cdots 
      \\
      \\
       \Omega_2^{w_1}(G_1)\ar[rr]^{{(f_1)}_\#}\ar[uu]^{\partial^{w_1}}&&
    \Omega_2^{w_2}(G_2)\ar[rr]^{{(f_2)}_\#}\ar[uu]^{\partial^{w_2}}&&
    \cdots  
      \\
      \\
       \cdots \ar[uu]^{\partial^{w_1}}&&
    \cdots \ar[uu]^{\partial^{w_2}}&&
 }
 \end{eqnarray*} 
  And by Remark~\ref{re-4.1}, we have a persistence complex
\begin{eqnarray*}
 \xymatrix{
\Gamma_0^{w_1}(G_1)\ar[rr]^{{(f_1)}_\#} &&
    \Gamma_0^{w_2}(G_2)\ar[rr]^{{(f_2)}_\#}&&
    \cdots 
     \\
     \\
 \Gamma_1^{w_1}(G_1)\ar[rr]^{{(f_1)}_\#}\ar[uu]^{\partial^{w_1}}&&
    \Gamma_1^{w_2}(G_2)\ar[rr]^{{(f_2)}_\#}\ar[uu]^{\partial^{w_2}}&&
    \cdots 
      \\
      \\
       \Gamma_2^{w_1}(G_1)\ar[rr]^{{(f_1)}_\#}\ar[uu]^{\partial^{w_1}}&&
    \Gamma_2^{w_2}(G_2)\ar[rr]^{{(f_2)}_\#}\ar[uu]^{\partial^{w_2}}&&
    \cdots  
      \\
      \\
       \cdots \ar[uu]^{\partial^{w_1}}&&
    \cdots \ar[uu]^{\partial^{w_2}}&&
 }
 \end{eqnarray*} 
 Each entry $\Omega_p^{w_n}(G_n)$ in the first persistence complex is a sub-$R$-module of the corresponding entry $\Gamma_p^{w_n}(G_n)$ in the second persistence complex, for any $p\geq 0$ and any  $n\geq 1$.  Hence  by  embedding each $\Omega_p^{w_n}(G_n)$ as a subspace of  $\Gamma_p^{w_n}(G_n)$ via the canonical inclusion $\iota_n$, we get a canonical inclusion $(\iota_1,\iota_2,\ldots)$ from the first persistence complex into the second persistence complex.

 By Theorem~\ref{th-4.3} and Remark~\ref{re-4.2},  for each $p\geq 0$, each  of the above persistence complexes induces  a persistence $R$-module  (we refer to \cite[Defiition~3.2]{[25-ZC05]} for the definition of persistence modules)
 \begin{eqnarray}
 H_p(G_1,w_1;R)\overset{(f_1)_*}{\longrightarrow}     \cdots  \overset{(f_{n-1})_*}{\longrightarrow} H_p(G_n,w_n;R)\overset{(f_n)_*}{\longrightarrow}H_p(G_{n+1},w_{n+1};R) \overset{(f_{n+1})_*}{\longrightarrow} \cdots
 \label{eq-5.1}
 \end{eqnarray}
 We call the persistence $R$-module (\ref{eq-5.1}) the persistent weighted path homology.  We notice that  (\ref{eq-5.1})  is of finite type since $V$ is a finite set.  

 \begin{theorem}\label{th-5.z}
 The canonical inclusion $(\iota_1,\iota_2,\ldots)$ of the first persistence complex $\Omega^w_*(G_*)$ into the second persistence complex $\Gamma^w_*(G_*)$  induces the identity map  from the persistence $R$-module (\ref{eq-5.1}) to itself.  
 \end{theorem}
 
 \begin{proof}
By Proposition~\ref{le-5.a},  the canonical inclusion 
\begin{eqnarray*}
\iota_n: \Omega^{w_n}_p(G_n)\longrightarrow \Gamma^{w_n}_p(G_n), ~~~ p\geq 0,
\end{eqnarray*}
 $n=1,2,\ldots$, induces an isomorphism of the weighted path homologies 
\begin{eqnarray*}
(\iota_n)_*: H_p(G_n,w_n;R)\overset{\cong}{\longrightarrow}  H_p(G_n,w_n;R), ~~~p\geq 0. 
\end{eqnarray*}
 Thus to prove Theorem~\ref{th-5.z},  it only suffices to prove that for any morphism of weighted digraphs
 \begin{eqnarray*}
 f_n: (G_n,w_n)\longrightarrow (G_{n+1},w_{n+1})
 \end{eqnarray*}
 where $G_n=(V_n,E_n)$ and $G_{n+1}=(V_{n+1},E_{n+1})$,  the restriction of the chain map 
 \begin{eqnarray}\label{eq-5.e1}
 (f_n)_{\#}: \{\mathcal{R}_{*}(V_n), \partial^{w_{n}}\}\longrightarrow \{\mathcal{R}_{*}(V_{n+1}), \partial^{w_{n+1}}\}
 \end{eqnarray}
 to the sub-chain complex $\Omega^{w_n}_*(G_n)$, and the restriction of (\ref{eq-5.e1}) to the sub-chain complex $\Gamma^{w_n}_*(G_n)$, induce the same homomorphism in  weighted path homologies   
 \begin{eqnarray}\label{eq-5.c6}
 (f_n)_*:   H_p(G_n,w_n;R){\longrightarrow}  H_p(G_{n+1},w_{n+1};R), ~~~p\geq 0. 
 \end{eqnarray}
We have a commutative diagram of chain complexes  and chain maps 
 \begin{eqnarray*}
 \xymatrix{
\Omega^{w_n}_*(G_n)  \ar[rrr]^{(f_n)_\#\mid_{\Omega^{w_n}_*(G_n)}} \ar[dd]^{\iota_n} &&&  \Omega^{w_{n+1}}_*(G_{n+1})\ar[dd]^{\iota_{n+1}}\\
\\
\Gamma^{w_n}_*(G_n)  \ar[rrr]^{(f_n)_\#\mid_{\Gamma^{w_n}_*(G_n)}}  &&&  \Gamma^{w_{n+1}}_*(G_{n+1}). 
 }
 \end{eqnarray*}
By taking  the  homologies in the  commutative diagram, since both $(\iota_n)_*$ and $(\iota_{n+1})_*$ are isomorphisms in homologies,  it follows that 
\begin{eqnarray*}
(f_n)_\#\mid_{\Omega^{w_n}_*(G_n)} ~~~ \text{ and  } ~~~  (f_n)_\#\mid_{\Gamma^{w_n}_*(G_n)}
\end{eqnarray*}
 induce the same homomorphism (\ref{eq-5.c6})  in weighted path homologies.  The proof follows.  
 \end{proof}
 
 \smallskip
 
 A particular family of sequences of morphisms of weighted digraphs is the {\it filtrations}.  Precisely,  a filtration of weighted digraphs is a sequence  of graphs 
 \begin{eqnarray}\label{eq-4.n1}
 G_1\subseteq G_2\subseteq \cdots \subseteq G_n\subseteq G_{n+1}\subseteq \cdots 
 \end{eqnarray}  
 where for any $n=1,2,\cdots$, $G_n=(V_n,E_n)$,    
 together with the non-vanishing functions 
 \begin{eqnarray*}
 w_\infty: \cup_{n=1}^\infty V_n \longrightarrow R
 \end{eqnarray*}
 and the induced weights
 \begin{eqnarray*}
 w_n=w_\infty\mid _{V_n}: V_n\longrightarrow R.
 \end{eqnarray*}  
 Here in (\ref{eq-4.n1}),  the subset relation  $G_n\subseteq G_{n+1}$ means  both  $V_n\subseteq V_{n+1}$ and $E_{n}\subseteq E_{n+1}$.  As particular cases, both of the persistence module (\ref{eq-5.1}) and Theorem~\ref{th-5.z} hold for  filtrations of weighted digraphs.

  \smallskip

 
 \subsection{Persistent homology with field coefficients is independent on weights}

Suppose $R$ is  a field $\mathbb{F}$.  Let  $n=1,2,\ldots$.   
Let $(G_n,w_n)$ be a finite or countable sequence of weighted digraphs,  together with a sequence of morphisms of weighted digraphs $f_n: (G_n,w_n)\longrightarrow (G_{n+1},w_{n+1})$.  We have the next theorem. 

\begin{theorem}\label{th-5.w}
The  persistent  weighted path homology 
 \begin{eqnarray}
 H_p(G_1,w_1;\mathbb{F})\overset{(f_1)_*}{\longrightarrow}     \cdots  \overset{(f_{n-1})_*}{\longrightarrow} H_p(G_n,w_n;\mathbb{F})\overset{(f_n)_*}{\longrightarrow}H_p(G_{n+1},w_{n+1};\mathbb{F}) \overset{(f_{n+1})_*}{\longrightarrow} \cdots
 \label{eq-5.11}
 \end{eqnarray}
 is independent on the choices  of the weights $w_n$, $n=1,2,\ldots$.  
\end{theorem}

We first prove some lemmas before proving Theorem~\ref{th-5.w}. 

\begin{lemma}\label{le-w.a}
Let $(G,w)$ be a weighted digraph with $G=(V,E)$.  Let $p\geq 0$.  Let $\partial_p: \mathcal{R}_p(V) \rightarrow \mathcal{R}_{p-1}(V)$ be the usual  boundary operator defined in (\ref{eq-2.3}) and $\partial^w_p:  \mathcal{R}_p^w(V) \rightarrow \mathcal{R}_{p-1}^w(V)$ be  the weighted boundary operator defined in (\ref{eq-3.13}), both with coefficients in $\mathbb{F}$.  Then as vector spaces over $\mathbb{F}$, 
\begin{enumerate}[(i).]
\item
 $\text{Ker}(\partial_p) \cong \text{Ker}(\partial^w_p)$;
 \item
  $\text{Im}(\partial_{p+1}) \cong \text{Im}(\partial^w_{p+1})$. 
 \end{enumerate} 
\end{lemma}

\begin{proof}
The proof of (i) is similar with \cite[Lemma~5.1]{27}.  Let $u\in \text{Ker}(\partial_p)$.  Then $u$ is an element in $\mathcal{R}_p(V)$ such that $\partial_p u=0$.  Suppose
\begin{eqnarray}\label{eq-5.q.1}
u=\sum_{k=1}^m a_k  e_{i_0^k\cdots  i_p^k}
\end{eqnarray}
where for each $1\leq k\leq m$, $a_k\in \mathbb{F}$, $i_0^k, \ldots, i_p^k$ are (not necessarily distinct) vertices   in  $V$, and $e_{i_0^k\cdots  i_p^k}$ is an elementary $p$-path on $V$.  We construct a map
\begin{eqnarray}\label{eq-5.q.3}
\varphi:  \text{Ker}(\partial_p)\longrightarrow \text{Ker}(\partial^w_p)
\end{eqnarray}
by setting 
\begin{eqnarray}\label{eq-5.q.2}
\varphi (u)=  \sum_{k=1}^{m}\frac{a_k}{w(i_0^k)\cdots w(i_p^k)} e_{i_0^k\cdots  i_p^k}. 
\end{eqnarray}
By a similar calculation with the proof of \cite[Lemma~5.1]{27},  it follows that $\varphi$ is well-defined, and  is an  isomorphism of vector spaces over $\mathbb{F}$.  Hence we obtain (i).   

The proof of (ii) is similar with \cite[Lemma~5.2]{27}.  We construct a   map
\begin{eqnarray}\label{eq-5.q.5}
\psi: \text{Im}(\partial_{p+1})\longrightarrow \text{Im}(\partial^w_{p+1})
\end{eqnarray}
by setting 
\begin{eqnarray*}
\psi(\partial_{p+1} (e_{i_0\cdots i_{p+1}}))=\partial_{p+1}^w (e_{i_0 \cdots i_{p+1}})
\end{eqnarray*}
for any elementary $(p+1)$-path $e_{i_0\cdots i_{p+1}}$ on $V$,  
and extending linearly over $\mathbb{F}$.   By a similar argument with  the proof of \cite[Lemma~5.2]{27},  it follows that  $\varphi$   is a linear isomorphism of vector spaces over $\mathbb{F}$.  Hence we obtain (ii).  
\end{proof}

\begin{lemma}\label{le-w.b}
Let $(G,w)$ be a weighted digraph with $G=(V,E)$.  Let $p\geq 0$.  Let $\partial_p: \mathcal{R}_p(V) \rightarrow \mathcal{R}_{p-1}(V)$ be the usual  boundary operator defined in (\ref{eq-2.3}) and $\partial^w_p:  \mathcal{R}_p^w(V) \rightarrow \mathcal{R}_{p-1}^w(V)$ be  the weighted boundary operator defined in (\ref{eq-3.13}), both with coefficients in $\mathbb{F}$.  Then as vector spaces over $\mathbb{F}$, 
\begin{enumerate}[(i).]

\item
$\mathcal{A}_p(G)\cap \text{Ker}(\partial_{p})\cong  \mathcal{A}_p(G)\cap\text{Ker}(\partial^w_{p})$;

\item
$\mathcal{A}_p(G)\cap \partial_{p+1}\mathcal{A}_{p+1}(G)\cong  \mathcal{A}_p(G)\cap \partial^w_{p+1}\mathcal{A}_{p+1}(G)$.   
\end{enumerate}
\end{lemma}

\begin{proof}
By (\ref{eq-5.q.1}) and (\ref{eq-5.q.2}), 
\begin{eqnarray*}
&&u \in \mathcal{A}_p(G)\cap \text{Ker}(\partial_{p})\\
 &\Longleftrightarrow&   e_{i_0^k\cdots  i_p^k}\in \mathcal{A}_p(G) \text{ for each } k=1,\ldots, m\\
&\Longleftrightarrow & 
\varphi(u) \in \mathcal{A}_p(G)\cap \text{Ker}(\partial^w_{p}). 
\end{eqnarray*}
Hence by restricting $\varphi$ in (\ref{eq-5.q.3}) to the subspace $\mathcal{A}_p(G)\cap \text{Ker}(\partial_{p})$,  we obtain an isomorphism of vector spaces over $\mathbb{F}$
\begin{eqnarray}\label{eq-5.u.1}
\varphi\mid _{\mathcal{A}_p(G)\cap \text{Ker}(\partial_{p})}:  \mathcal{A}_p(G)\cap \text{Ker}(\partial_{p})\overset{\cong}{\longrightarrow}  \mathcal{A}_p(G)\cap\text{Ker}(\partial^w_{p}). 
\end{eqnarray}
 We obtain (i).  
 On the other hand,  for any elementary $(p+1)$-path $e_{i_0\cdots i_{p+1}}$ on $V$,  
 \begin{eqnarray*}
&& \partial_{p+1} (e_{i_0\cdots i_{p+1}}) \in \mathcal{A}_p(G) \\
 &\Longleftrightarrow&  e_{i_0\cdots\widehat{i_q} \cdots i_{p+1}}\in \mathcal{A}_p(G) \text{ for each } q=0,\ldots, p+1\\
 &\Longleftrightarrow& \partial^w_{p+1} (e_{i_0\cdots i_{p+1}}) \in \mathcal{A}_p(G). 
\end{eqnarray*}
Moreover, $\psi$ in (\ref{eq-5.q.5}) sends $\partial_{p+1}\mathcal{A}_{p+1}(G)$ isomorphically to $\partial_{p+1}^w\mathcal{A}_{p+1}(G)$. 
 Hence by restricting $\psi$ in (\ref{eq-5.q.5}) to the subspace $\mathcal{A}_p(G)\cap \partial_{p+1}\mathcal{A}_{p+1}(G)$,  we obtain an isomorphism of vector spaces over $\mathbb{F}$
 \begin{eqnarray}\label{eq-5.v.1}
 \psi\mid_{\mathcal{A}_p(G)\cap \partial_{p+1}\mathcal{A}_{p+1}(G)}:  \mathcal{A}_p(G)\cap \partial_{p+1}\mathcal{A}_{p+1}(G)\overset{\cong}{\longrightarrow}  \mathcal{A}_p(G)\cap \partial^w_{p+1}\mathcal{A}_{p+1}(G).   
 \end{eqnarray}
 We obtain (ii).  
\end{proof}

The proof of Theorem~\ref{th-5.w} follows from (\ref{eq-4.00}), Lemma~\ref{le-w.a} and Lemma~\ref{le-w.b}.  

\begin{proof}[Proof of Theorem~\ref{th-5.w}]
Suppose $G_n=(V_n,E_n)$ for each $n\geq 1$.   Let $\partial_*(n)$ and $\partial^w_*(n)$ be the usual (unweighted) boundary operators and the weighted boundary operators with respect to $w_n$ respectively,  for the chain complex $\mathcal{R}_*(V_n)$.  For each $p\geq 0$,  we have the following two commutative diagrams
\begin{eqnarray*}
\xymatrix{
\mathcal{A}_p(G_n)\cap \text{Ker}(\partial_p(n))\ar[dd]_{\varphi\mid_{\mathcal{A}_p(G_n)\cap \text{Ker}(\partial_p(n))}}\ar[r]^{(f_n)_\#~~~}& \mathcal{A}_p(G_{n+1})\cap \text{Ker}(\partial_p(n+1))\ar[dd]^{\varphi\mid_{\mathcal{A}_p(G_{n+1})\cap \text{Ker}(\partial_p(n+1))}}\\
\\
\mathcal{A}_p(G_n)\cap \text{Ker}(\partial^w_p(n))\ar[r]^{(f_n)_\#~~~}& \mathcal{A}_p(G_{n+1})\cap \text{Ker}(\partial^w_p(n+1))
}
\end{eqnarray*}
and 
\begin{eqnarray*}
\xymatrix{
\mathcal{A}_p(G_n)\cap \partial_{p+1}(n)\mathcal{A}_{p+1}(G_n)\ar[dd]_{\psi\mid_{\mathcal{A}_p(G_n)\cap \partial_{p+1}(n)\mathcal{A}_{p+1}(G_n)}}\ar[r]^{(f_n)_\#~~~~}& \mathcal{A}_p(G_{n+1})\cap \partial_{p+1}(n+1)\mathcal{A}_{p+1}(G_{n+1})\ar[dd]^{\psi\mid_{\mathcal{A}_p(G_{n+1})\cap \partial_{p+1}(n+1)\mathcal{A}_{p+1}(G_{n+1})}}\\
\\
\mathcal{A}_p(G_n)\cap \partial^w_{p+1}(n)\mathcal{A}_{p+1}(G_n)\ar[r]^{(f_n)_\#~~~~}& \mathcal{A}_p(G_{n+1})\cap \partial^w_{p+1}(n+1)\mathcal{A}_{p+1}(G_{n+1}).
}
\end{eqnarray*}
Therefore,  for any $p\geq 0$,  we have the induced commutative diagram 
\begin{eqnarray*}
\xymatrix{
H_p(G_n;  \mathbb{F}) \ar[dd]_{ \text{isomorphism}\atop  \text{induced by }(\varphi,\psi)}^{\cong} \ar[rr]^{(f_n)_*}&&H_p(G_{n+1};  \mathbb{F})\ar[dd]^{ \text{isomorphism}\atop \text{ induced by }(\varphi,\psi)}_{\cong}\\\\
H_p(G_n,w_n;  \mathbb{F})\ar[rr]^{(f_n)_*} && H_p(G_{n+1},w_{n+1};  \mathbb{F}). 
}
\end{eqnarray*}
Consequently, the weighted path homologies $H_p(G_n,w_n;  \mathbb{F})$ and $H_p(G_{n+1},w_{n+1};  \mathbb{F})$ as well as the homomorphism $(f_n)_*$ do not depend on the choices of $w_n$ and $w_{n+1}$.  Theorem~\ref{th-5.w} follows. 
\end{proof}

Let  all the morphisms $f_n$, $n=1,2,\ldots$,  be the identity map in Theorem~\ref{th-5.w}.  Then the persistent weighted path homology (\ref{eq-5.11}) reduces to the usual weighted path homology.  We obtain the next corollary. 

\begin{corollary}\label{co-5.1}
Let $G$ be a vertex weighted directed graph with all weights of vertices nonzero, no loops or double-directed edges. Let $\partial:\Omega_{n+1} \rightarrow \Omega_{n}$ and $\partial^w:\Omega_{n+1}^w \rightarrow \Omega_{n}^w$ ($n\geqslant 0$)denote the usual unweighted boundary operator
and the weighted boundary operator respectively. Then $H_n(G) \cong H_n(G,w)$ ($n\geqslant 0$) as $R=\mathbb{F}$ is a field.  \qed
\end{corollary}

\subsection{Some examples} \label{ss4.3}
 
 In Theorem~\ref{th-5.w}, it is proved that the persistent weighted path homology with coefficients in a field $\mathbb{F}$ does not depend on the weights.  Nevertheless,  if we choose the coefficients in a general ring $R$,  then the persistent weighted path homology will depend on the weights.  In this subsection,  we give some examples to illustrate the effects of the weights on the persistent weighted path homology. 
 
 \smallskip

 \begin{example}\label{ex-6.1}
 Let $V=\{i_0, i_1, i_2\}$, $E_1=\{(i_0, i_1)\}$ and   $E_2=\{(i_0, i_1), (i_1, i_2)\}$.  Let $G_1=(V,E_1)$ and $G_2=(V,E_2)$.   A  weight on both $G_1$ and $G_2$ is given by a non-vanishing function $w: V\longrightarrow \mathbb{Z}\setminus \{0\}$.   We take the morphism of weighted digraphs as the canonical inclusion $\iota: (G_1,w)\longrightarrow (G_2,w)$.  
 
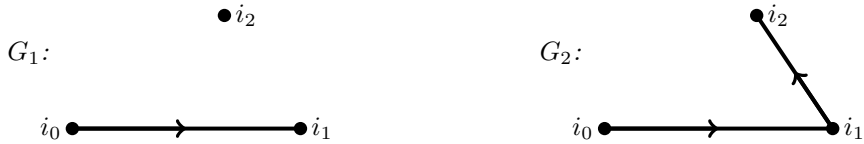
 \begin{center}
\begin{tikzpicture}[line width=1.5pt]
\coordinate [label=left:$i_0$]  (A) at (2,0); 
\coordinate [label=right:$i_1$]  (B) at (5,0); 
   \coordinate [label=right:$G_1$:]  (F) at (1,1); 
 \draw [line width=1.5pt] (A) -- (B);
 \coordinate [label=right:$i_2$]  (C) at (4,1.5); 
 \draw[->] (2,0) -- (3.5,0);
  \coordinate [label=left:$i_0$]  (A1) at (9,0); 
  \coordinate [label=right:$i_1$]  (B1) at (12,0); 
  \coordinate [label=right:$G_2$:]  (F1) at (8,1); 
 \draw [line width=1.5pt] (A1) -- (B1);
 \coordinate [label=right:$i_2$]  (C1) at (11,1.5); 
  \draw[->] (9,0) -- (10.5,0);
 \draw[line width=1.5pt] (12,0) -- (11,1.5);
 \draw[->] (12,0) -- (11.5,0.75);
\fill (2,0) circle (2.5pt) (5,0) circle (2.5pt)  (4,1.5) circle (2.5 pt);        
\fill (9,0) circle (2.5pt) (12,0) circle (2.5pt)  (11,1.5) circle (2.5 pt);        
\end{tikzpicture}
\end{center}
We have
\begin{eqnarray*}
\mathcal{R}_0(V)&=&\mathbb{Z}( e_{i_0}, e_{i_1}, e_{i_2}),\\
\mathcal{R}_1(V)&=&\mathbb{Z}(\{ e_{i_ai_b}\mid  0\leq a,b\leq 2, a\neq b\}),\\
\mathcal{R}_2(V)&=&\mathbb{Z}(\{ e_{i_ai_bi_c}\mid  0\leq a,b,c\leq 2, a,b,c \text{ are distinct}\}). 
\end{eqnarray*}
 Here $\mathbb{Z}(*)$ denotes the free $\mathbb{Z}$-module generated by the set $*$.  The weighted boundary operator is given by
 \begin{eqnarray*}
& \partial_0^w(e_{i_0})=  \partial_0^w(e_{i_1})=  \partial_0^w(e_{i_2})=0,\\
& \partial_1^w(e_{i_ai_b})= w(i_a) e_{i_b} - w({i_b}) e_{i_a},\\
& \partial_2^w(e_{i_ai_bi_c})=w({i_a})e_{i_bi_c}-w({i_b})e_{i_ai_c} + w({i_c}) e_{i_ai_b}. 
 \end{eqnarray*} 
 Moreover, 
 \begin{eqnarray*}
 \mathcal{A}_0(G_1)&=&\mathbb{Z}( e_{i_0}, e_{i_1}, e_{i_2}),\\
 \mathcal{A}_1(G_1)&=&\mathbb{Z}(e_{i_0i_1}),\\
 \mathcal{A}_2(G_1)&=&0
 \end{eqnarray*}
 and 
  \begin{eqnarray*}
 \mathcal{A}_0(G_2)&=&\mathbb{Z}( e_{i_0}, e_{i_1}, e_{i_2}),\\
 \mathcal{A}_1(G_2)&=&\mathbb{Z}(e_{i_0i_1},e_{i_1i_2}),\\
 \mathcal{A}_2(G_2)&=&\mathbb{Z}(e_{i_0i_1i_2}). 
 \end{eqnarray*}
 Consequently,  
  \begin{eqnarray*}
&& \Omega^w_0(G_1)=\Gamma^w_0(G_1)=\mathbb{Z}( e_{i_0}, e_{i_1}, e_{i_2}),\\
 &&\Omega^w_1(G_1)= \Gamma^w_1(G_1)=\mathbb{Z}(e_{i_0i_1}),\\
 && \Omega^w_2(G_1)= \Gamma^w_2(G_1)=0
 \end{eqnarray*}
 and 
   \begin{eqnarray*}
&& \Omega^w_0(G_2)=\Gamma^w_0(G_2)=\mathbb{Z}( e_{i_0}, e_{i_1}, e_{i_2}),\\
 &&\Omega^w_1(G_2)= \mathbb{Z}(e_{i_0i_1},e_{i_1i_2}),\\
 &&\Gamma^w_1(G_2)=\mathbb{Z}(e_{i_0i_1},e_{i_1i_2}, w(i_0)e_{i_1i_2}-w(i_1)e_{i_0i_2}+w(i_2)e_{i_0i_1})\\
 &&~~~~~~~~~=\mathbb{Z}(e_{i_0i_1},e_{i_1i_2}, w(i_1)e_{i_0i_2}),\\
 && \Omega^w_2(G_2)=0,\\
 && \Gamma^w_2(G_2)=\mathbb{Z}(e_{i_0i_1i_2}). 
 \end{eqnarray*}
 Therefore,  
 \begin{eqnarray}
 && H_0(G_1,w;\mathbb{Z})= \mathbb{Z}( e_{i_0}, e_{i_1}, e_{i_2})/ \mathbb{Z}(w(i_0)e_{i_1}-w(i_1)e_{i_0})\nonumber\\
 &&~~~~~~~~~~~~~~~~\cong\mathbb{Z}\oplus\mathbb{Z}\oplus (\mathbb{Z}/\text{gcd}\{w(i_0),w(i_1)\}); 
 \label{eq-sl}\\
 && H_1(G_1,w;\mathbb{Z})=  H_2(G_1,w;\mathbb{Z}) =0 \nonumber
 \end{eqnarray}
 and 
  \begin{eqnarray*}
 && H_0(G_2,w;\mathbb{Z})= \mathbb{Z}( e_{i_0}, e_{i_1}, e_{i_2})/ \mathbb{Z}(w(i_0)e_{i_1}-w(i_1)e_{i_0},w(i_1)e_{i_2}-w({i_2}) e_{i_1}); \\
 && H_1(G_2,w;\mathbb{Z})=  H_2(G_2,w;\mathbb{Z}) =0. 
 \end{eqnarray*}
 The induced homomorphism of $\iota$  in the $0$-dimensional weighted path homology is the epimorphism modulo the $\mathbb{Z}$-module  $\mathbb{Z}(w(i_1)e_{i_2}-w({i_2}) e_{i_1})$. 
 \qed
  \end{example}

 The isomorphism (\ref{eq-sl}) is obtained as follows. We choose integers $a$ and $b$ such that  
 \begin{eqnarray*}
 a\frac{w(i_0)}{\text{gcd}\{w(i_0),w(i_1)\}}+ b \frac{w(i_1)}{\text{gcd}\{w(i_0),w(i_1)\}}=1.
 \end{eqnarray*} 
 Then we have the invertible matrix
 \begin{eqnarray*}
 \begin{bmatrix}
\frac{w(i_0)}{\text{gcd}\{w(i_0),w(i_1)\}} & -\frac{w(i_1)}{\text{gcd}\{w(i_0),w(i_1)\}} \\
b & a
\end{bmatrix}
\in \text{SL}(2,\mathbb{Z}). 
 \end{eqnarray*}
Hence 
\begin{eqnarray*}
\begin{bmatrix}
e_{i_1} \\
e_{i_0}
\end{bmatrix}
=\begin{bmatrix}
\frac{w(i_0)}{\text{gcd}\{w(i_0),w(i_1)\}} & -\frac{w(i_1)}{\text{gcd}\{w(i_0),w(i_1)\}} \\
b & a
\end{bmatrix}^{-1}
\begin{bmatrix}
\frac{w(i_0)}{\text{gcd}\{w(i_0),w(i_1)\}}  e_{i_1}- \frac{w(i_1)}{\text{gcd}\{w(i_0),w(i_1)\}} e_{i_0} \\
 be_{i_1} +a e_{i_0}
 \end{bmatrix}. 
\end{eqnarray*}
It follows that 
\begin{eqnarray*}
\mathbb{Z}(e_{i_0},e_{i_1})=\mathbb{Z}\Big(\frac{w(i_0)}{\text{gcd}\{w(i_0),w(i_1)\}}  e_{i_1}- \frac{w(i_1)}{\text{gcd}\{w(i_0),w(i_1)\}} e_{i_0}, 
 be_{i_1} +a e_{i_0}\Big). 
\end{eqnarray*}
Hence we obtain (\ref{eq-sl}). 

\begin{example}
Let $(G_1,w)$ and $(G_2,w)$ be the weighted digraphs given in Example~\ref{ex-6.1}.  Suppose in addition that $w(i_1)=w(i_2)$.  We consider a morphism of weighted digraphs $f: (G_2,w)\longrightarrow (G_1,w)$ given by
\begin{eqnarray*}
f(i_0)=i_0,  ~~~ f(i_1)=f(i_2)=i_1. 
\end{eqnarray*}
Then  the  induced homomorphism of $f$  in the $0$-dimensional weighted path homology is the canonical epimorphism 
\begin{eqnarray*}
f_*: && \mathbb{Z}( e_{i_0}, e_{i_1}, e_{i_2})/ \mathbb{Z}(w(i_0)e_{i_1}-w(i_1)e_{i_0},w(i_1)e_{i_2}-w({i_2}) e_{i_1})\\
&&\longrightarrow   \mathbb{Z}( e_{i_0}, e_{i_1})/ \mathbb{Z}(w(i_0)e_{i_1}-w(i_1)e_{i_0}). 
\end{eqnarray*}
\qed
\end{example}

 \section{Weighted path homologies for joins of weighted digraphs and persistence}\label{sss5}
 
In this section,  we prove the product rule for concatenations of weighted paths. Consequently, we give a K\"{u}nneth-type formula for the  weighted path homologies for joins of weighted digraphs in Theorem~\ref{th-7.9}, as well as  a persistent version in Theorem~\ref{th-17.9}.  Throughout this  section,  we assume that $R$ is a principal ideal domain.  
 
 \smallskip
 
 \subsection{Weighted path homologies for joins of weighted digraphs}
 
Let $(G,w)$ and $(G',w')$ be two weighted digraphs where $G=(V,E)$ and $G'=(V',E')$. Suppose $V$ and $V'$ are disjoint. Then $E$ and $E'$ are disjoint as well.  We define the {\it join} of $(G,w)$ and $(G',w')$ as the weighted digraph $(G*G',w*w')$ 
by  the followings:  
\begin{itemize}
\item
the set of vertices of the digraph $G*G'$ is $ V\sqcup V'$; 
\item
the set of directed edges of the digraph $G*G'$ is $E\sqcup E'\sqcup\{(i,j)\mid i\in V, j\in V'\}$; 
\item
the weight $w*w'$ on the digraph $G*G'$  is given by 
\begin{eqnarray}\label{eq-17.2}
(w*w')(x)=\left\{
\begin{array}{cc}
w(x), & \text{ if  } x\in V, \\
w'(x), & \text{ if } x\in V'. 
\end{array}
\right. 
\end{eqnarray}
\end{itemize}
We point that our definition of joins of weighted digraphs is just a weighted version of \cite[Definition~6.1]{9}.

 Let $e_{i_0\cdots i_p}$ be an  elementary $p$-path on $V$ and $e_{j_0\cdots j_q}$ be an   elementary $q$-path on $V'$.  Then we have a $(p+q+1)$-path $e_{i_0\cdots i_p j_0\cdots j_q}$ on $V\sqcup V'$, which will be alternatively denoted as   $e_{i_0\cdots i_p} e_{j_0\cdots j_q}$, and called the {\it concatenation}.  By extending the concatenation bilinearly over the ring $R$,  we can define the concatenation as an  $R$-linear map
 \begin{eqnarray}\label{eq-7.1}
\mu:  \Lambda_p(V)\otimes \Lambda_q(V')\longrightarrow \Lambda_{p+q+1}(V\sqcup V') 
 \end{eqnarray}
 where the tensor product is over $R$.  
It can be proved that  $\mu$ in (\ref{eq-7.1}) is injective.  The next lemma gives a product rule, which is a weighted version of \cite[Lemma~2.6]{9}.

\begin{lemma}\label{le-3.3}
Let $p,q\geq -1$. If $u\in \Lambda_p(V)$ and $v\in \Lambda_q(V')$, then 
\begin{eqnarray}\label{eq-3.23}
\partial^{w*w'}(uv)=(\partial^w u) v + (-1)^{p+1} u(\partial^{'w'} v). 
\end{eqnarray}
Here $uv$ is the concatenation of $u$ and $v$. 
\end{lemma}
\begin{proof}
It suffices to prove (\ref{eq-3.23}) for $u=e_{i_0\cdots i_p}$ and $v=e_{j_0\cdots j_q}$.  By a similar calculation with the proof of \cite[Lemma~2.6]{9},
\begin{eqnarray*}
\partial^{w*w'} (e_{i_0\cdots i_p} e_{j_0\cdots j_q})&=& \partial^{w*w'}  (e_{i_0\cdots i_p j_0\cdots j_q})\\
&=& \sum_{r=0}^{p} (-1)^r w(i_r) e_{i_0 \cdots \widehat{i_r} \cdots  i_p  j_0\cdots j_q} +   \sum_{r=0}^{q} (-1)^{p+r+1} w'(j_r) e_{i_0  \cdots  i_p  j_0\cdots \widehat{j_r}\cdots j_q}\\
&=& (\partial^w e_{i_0\cdots i_p}) e_{j_0\cdots j_q} + (-1)^{p+1} e_{i_0\cdots i_p}(\partial^{'w'} e_{j_0\cdots j_q}). 
\end{eqnarray*}
Hence (\ref{eq-3.23}) follows. 
\end{proof}

The next lemma is a weighted version of \cite[formula~(6.3)]{9}. 

\begin{lemma}\label{le-7.1}
For any $r\geq -1$, we have that as  $R$-modules, 
\begin{eqnarray}\label{eq-881}
\Omega_r^{w*w'}(G*G')\cong \bigoplus_{p,q\geq -1,\atop p+q=r-1}\big(\Omega^w_p(G)\otimes \Omega^{w'}_q(G')\big).  
\end{eqnarray}
Moreover, as chain complexes,
\begin{eqnarray}\label{eq-882}
\big(\Omega_*^w(G)\otimes \Omega_*^{w'}(G')\big)_{r-1}\cong  \Omega_r^{w*w'}(G*G'),  ~~~r\geq 0. 
\end{eqnarray}
\end{lemma}
\begin{proof}
The proof is similar with the proof of \cite[Proposition~6.4 and Theorem~6.5]{9}.  For any $p,q\geq -1$ with $p+q=r-1$,   the concatenation (\ref{eq-7.1}) induces a map
\begin{eqnarray}\label{eq-7.a}
\mu_\#:  \Omega^w_p(G)\otimes \Omega^{w'}_q(G')\longrightarrow \Omega_r^{w*w'}(G*G'). 
\end{eqnarray}
It can be proved that  $\mu_\#$ in (\ref{eq-7.a}) is injective, and gives an isomorphism
\begin{eqnarray}\label{eq-7.e}
\mu_\#:  \bigoplus_{p,q\geq -1,\atop p+q=r-1}\big(\Omega^w_p(G)\otimes \Omega^{w'}_q(G')\big) \overset{\cong}{\longrightarrow}\Omega_r^{w*w'}(G*G'). 
\end{eqnarray}
The $R$-module isomorphism (\ref{eq-881})   is obtained.  Moreover,  by applying (\ref{eq-3.23}) in Lemma~\ref{le-3.3}, the map $\mu_\#$ in (\ref{eq-7.e}) commutes with   the boundary operators of the tensor product chain complex $ \Omega^w_*(G)\otimes \Omega^{w'}_*(G')$ (cf. \cite[Proposition~3B.1]{hatcher}) and the boundary operators of the chain complex $\Omega_{*}^{w*w'}(G*G')$ by shifting up one dimension
\begin{eqnarray*}
\mu_\#\circ (\partial^w\otimes \partial^{'w'}) = \partial^{w*w'}\circ \mu_\#. 
\end{eqnarray*}
  Hence we obtain the isomorphism (\ref{eq-882}) of chain complexes. 
\end{proof}

The next theorem follows from the algebraic K\"{u}nneth formula \cite[theorem~3B.5]{hatcher} and Lemma~\ref{le-7.1}.

\begin{theorem}\label{th-7.9}
For each $r\geq 0$, there is a natural short exact sequence 
\begin{eqnarray*}
&0\longrightarrow \bigoplus_{p,q\geq -1,\atop p+q=r-1}\big(H_p(G,w;R)\otimes H_q(G',w';R)\big) \longrightarrow H_{r}(G*G',w*w'; R)\\
&\longrightarrow  \bigoplus_{p,q\geq -1,\atop p+q=r-1} \text{Tor}_R\big(H_p(G,w;R),H_{q-1}(G',w';R)\big)\longrightarrow 0
\end{eqnarray*}
and this sequence splits. 
\end{theorem} 
\begin{proof}
We consider the chain complexes $\{\Omega_*^w(G),\partial^w\}$, $\{\Omega_*^{w'}(G'),\partial^{'w'}\}$ and their tensor product.  
By  \cite[theorem~3B.5]{hatcher},  we have a short exact sequence
\begin{eqnarray}\label{eq-7.9a}
&0\longrightarrow \bigoplus_{p,q\geq -1,\atop p+q=r-1}\big(H_p(\{\Omega_*^w(G),\partial^w\})\otimes H_q(\{\Omega_*^{w'}(G'),\partial^{'w'}\})\big)\nonumber\\
& \longrightarrow H_{r-1}(\{\Omega_*^w(G)\otimes \Omega_*^{w'}(G'),\partial^w\otimes \partial^{'w'}\})
\longrightarrow \nonumber\\
& \bigoplus_{p,q\geq -1,\atop p+q=r-1} \text{Tor}_R\big(H_p(\{\Omega_*^w(G),\partial^w\}),H_{q-1}(\{\Omega_*^{w'}(G'),\partial^{'w'}\})\big)\longrightarrow 0
\end{eqnarray}
and this sequence splits.   
On the other hand,  by 
 taking the homology groups of the chain complexes   on both sides of (\ref{eq-882}),  we have 
\begin{eqnarray}\label{eq-7.9b}
H_{r-1}(\{\Omega_*^w(G)\otimes \Omega_*^{w'}(G'),\partial^w\otimes \partial^{'w'}\})\cong H_r(G*G',w*w'; R). 
\end{eqnarray}
The theorem follows from (\ref{eq-7.9a}) and (\ref{eq-7.9b}). 
\end{proof}

\smallskip

\subsection{The persistence of Theorem~\ref{th-7.9}}

Let  $n=1,2,\ldots$.   
 We consider two finite or countable sequences of weighted digraphs $(G_n,w_n)$ and $(G'_n,w'_n)$, together with two sequences of morphisms of weighted digraphs $f_n: (G_n,w_n)\longrightarrow (G_{n+1},w_{n+1})$ and $f'_n: (G'_n,w'_n)\longrightarrow (G'_{n+1},w'_{n+1})$ respectively.   Suppose for each $n$,  $G_n=(V_n,E_n)$ and $G'=(V'_n,E'_n)$ such that $V_n$ and $V'_n$ are disjoint.  Then we have an induced sequence of weighted digraphs
 $(G_n*G'_n,w_n*w'_n)$, together with an induced sequence of morphisms of weighted digraphs
 \begin{eqnarray*}
 f_n*f'_n: (G_n*G'_n,w_n*w'_n)\longrightarrow (G_{n+1}*G'_{n+1},w_{n+1}*w'_{n+1})
 \end{eqnarray*}
 given by
 \begin{eqnarray}\label{eq-17.1}
  (f_n*f'_n)(x)=\left\{
\begin{array}{cc}
f_n(x), & \text{ if  } x\in V_n, \\
f'_n(x), & \text{ if } x\in V'_n. 
\end{array}
\right. 
 \end{eqnarray}
It follows from Definition~\ref{def-4.2},  (\ref{eq-17.2}) and (\ref{eq-17.1})  that 
\begin{eqnarray*}
(w_{n+1}*w'_{n+1})\big((f_n*f'_n)(x)\big)= (w_n*w'_n)(x)
\end{eqnarray*}
for any $x\in V_n\sqcup V'_n$.  Moreover,  
 it can be verified that for any directed edge of $G_n*G'_n$, its image under $f_n*f'_n$ is a directed edge of $G_{n+1}*G'_{n+1}$.  Hence $ f_n*f'_n$ is a morphism of weighted digraphs.  
 
 \begin{theorem}\label{th-17.9}
For each $r\geq 0$, there is a commutative diagram in which each row is a natural short exact sequence and each sequence splits  
\begin{eqnarray*}
\xymatrix{
0\ar[r] &\bigoplus_{p,q\geq -1,\atop p+q=r-1}\big(H_p(G_1,w_1;R)\otimes H_q(G'_1,w'_1;R)\big) \ar[dd]^{(f_1*f'_1)_*}\ar[r] & H_{r}(G_1*G'_1,w_1*w'_1; R)\ar[dd]^{(f_1*f'_1)_*}\\ 
\\
0\ar[r] &\bigoplus_{p,q\geq -1,\atop p+q=r-1}\big(H_p(G_2,w_2;R)\otimes H_q(G'_2,w'_2;R)\big) \ar[dd]^{(f_2*f'_2)_*}\ar[r] & H_{r}(G_2*G'_2,w_2*w'_2; R)\ar[dd]^{(f_2*f'_2)_*}\\ 
\\
&\cdots & \cdots \\ 
 \ar[r] &  \bigoplus_{p,q\geq -1,\atop p+q=r-1} \text{Tor}_R\big(H_p(G_1,w_1;R),H_{q-1}(G'_1,w'_1;R)\big)\ar[dd]^{(f_1*f'_1)_*}\ar[r] & 0. \\
 \\
  \ar[r] &  \bigoplus_{p,q\geq -1,\atop p+q=r-1} \text{Tor}_R\big(H_p(G_2,w_2;R),H_{q-1}(G'_2,w'_2;R)\big)\ar[dd]^{(f_2*f'_2)_*}\ar[r] & 0. \\
  \\
  &  \cdots & 
}
\end{eqnarray*}
\end{theorem} 

\begin{proof}
By the naturalities of the weighted path homology and the $\text{Tor}$ functor,  the commutative diagram follows from Theorem~\ref{th-7.9}. 
\end{proof}
 
  \bigskip

 \noindent {\bf Acknowledgement}.  This research is supported by the National Science Foundation of China (No.11671401), Scientific Research Projects of Hebei Education Department (No.QN2019333) and Natural Fund Project of Cangzhou Science and Technology Bureau (No.177000002) .

\bigskip

\noindent Chong Wang (Corresponding author)

\noindent Address:   $^1$School of Mathematics, Renmin University of China, Beijing 100872, China. $^2$School of Mathematics and Statistics, Cangzhou Normal University, Cangzhou 061000, China.

\noindent e-mail: wangchong\_618@163.com 

 \bigskip
 
\noindent Shiquan Ren

\noindent Address: Yau Mathematical Sciences Center, Tsinghua University, China 100084.

\noindent e-mail: srenmath@126.com

 \bigskip
 
\noindent Jie Wu

\noindent Address: School of Mathematics and Information Science, Hebei Normal University,
China 050024.

\noindent e-mail: wujie@hebtu.edu.cn

 \bigskip
 
\noindent Yong Lin

\noindent Address:  $^1$Yau Mathematical Sciences Center, Tsinghua University, Beijing 100872, China.
$^2$College of Mathematics and Informatics, Fujian Normal University, Fuzhou 350000, China.

\noindent e-mail: yonglin@tsinghua.edu.cn

}

\end{document}